\documentclass{amsart}
\usepackage{amscd,amsfonts,amssymb,amsmath}
\usepackage[margin=4.2cm]{geometry}

\usepackage{amsmath}

\newtheorem{theorem}{Theorem}[section]
\newtheorem{cor}[theorem]{Corollary}
\newtheorem{lemma}[theorem]{Lemma}
\newtheorem{prop}[theorem]{Proposition}

\theoremstyle{definition}

\theoremstyle{remark}

\numberwithin{equation}{subsection}
\theoremstyle{plain}

\newtheorem{question}{Question}

\def\Z{\mathbb Z}
\def \N { {\rm N}}

\def \pw{{\rm pw}}

\newcommand{\secref}[1]{Section~\ref{#1}}
\newcommand{\thmref}[1]{Theorem~\ref{#1}}
\newcommand{\lemref}[1]{Lemma~\ref{#1}}

\newcommand{\corref}[1]{Corollary~\ref{#1}}

\newcommand{\eqnref}[1]{~{\textrm(\ref{#1})}}
\numberwithin{equation}{section}

\begin{document}

\title[PALINDROMIC WIDTH OF FINITELY GENERATED SOLVABLE GROUPS]{PALINDROMIC WIDTH OF FINITELY GENERATED SOLVABLE GROUPS}

\author{Valeriy ~G.~Bardakov}
\address{Sobolev Institute of Mathematics, Novosibirsk State University, Novisibirsk 630090, Russia}
\address{and}
 \address{Laboratory of Quantum Topology, Chelyabinsk State University, Brat’evKashirinykh street 129, Chelyabinsk 454001, Russia}
\email{bardakov@math.nsc.ru}
\author{Krishnendu Gongopadhyay}
\address{Department of Mathematical Sciences, Indian Institute of Science Education and Research (IISER) Mohali,
Knowledge City, Sector 81, S.A.S. Nagar, P.O. Manauli 140306, India}
\email{krishnendu@iisermohali.ac.in, krishnendug@gmail.com}

\thanks{The authors gratefully acknowledge the support of the Indo-Russian DST-RFBR project grant DST/INT/RFBR/P-137}
\thanks{Bardakov is partially supported by Laboratory of Quantum Topology of Chelyabinsk State University(Russian Federation government grant 14.Z50.31.0020) }

\date{\today}


\begin{abstract}
We investigate the palindromic width of finitely generated solvable groups. We prove  that every finitely
generated $3$-step solvable group has finite palindromic width. More generally, we show the finiteness of
palindromic width for finitely generated abelian-by-nilpotent-by-nilpotent groups. For arbitrary solvable
groups of step $\geq 3$, we prove that if  $G$ is a finitely generated solvable group that is an extension of an abelian group by a group satisfying the maximal condition for normal subgroups, then the palindromic width of $G$ is finite. We also prove that the palindromic width of $Z \wr \Z$ with respect to the set of standard  generators is $3$.
\end{abstract}
\maketitle

\section{Introduction}
Let $G$ be a group with a set of generators $X$. A reduced word in the alphabet $X^{\pm 1}$
is a \emph{palindrome} if it reads the same forwards and backwards.
The palindromic length $l_{\mathcal P}(g)$
of an element $g$ in $G$ is the minimum number $k$ such that $g$ can be expressed as a product of $k$
palindromes. The \emph{palindromic width} of $G$ with respect to $X$ is defined to
be ${\rm pw}(G, X) = \underset{g \in G}{\sup} \ l_{\mathcal{P}}(g)$. When there is no confusion about the underlying generating set $X$, we simply denote the palindromic width with respect to $X$ by ${\rm pw}(G)$.
Palindromes in groups have already proved useful in studying various aspects
of combinatorial group theory and geometry, for example see  \cite{BST}--\cite{de2}, \cite{gk1}--\cite{hel},  \cite{P}.

\medskip  For $g$, $h$ in $G$, the \emph{commutator} of $g$ and $h$ is defined as $[g,
h]=g^{-1} h^{-1}gh$. If $\mathcal{C}$ is the set of commutators in some group
$G$ then the commutator  subgroup $G'$ is generated by $\mathcal{C}$.
The \emph{commutator length} $l_{\mathcal{C}}(g)$ of an element $g \in G'$ is the minimal number $k$ such that $g$ can be expressed as a product of $k$ commutators.  The \emph{commutator width} of $G$ is defined by $\underset{g \in G}{\sup} \ l_{\mathcal{\mathcal C}}(g)$ and is
denoted by ${\rm cw}(G)$. It is well known \cite{R} that the commutator width of a
free non-abelian group is infinite, but the commutator width of a finitely
generated nilpotent group is finite (see \cite{AR,AR1}). An algorithm of the
computation of the commutator length in free non-abelian groups can be found in
\cite{B}.

\medskip In \cite{BG}, we initiated the investigation of finding the palindromic width of a finitely generated
group that is free in some variety of groups. We proved that the palindromic width of a finitely generated
free nilpotent group is finite. Further using a result of Akhavan-Malayeri and Rhemtulla \cite{MR}, we proved
that $\pw(F_n(\mathcal U \mathcal N), X) \leq 5n$, where $\mathcal U$ is the variety of abelian groups,
 $\mathcal N$ is the variety of nilpotent groups and $X=\{x_1, \ldots, x_n\}$ is a basis of the group
 $F_n(\mathcal U \mathcal N)$ that is free in $\mathcal U \mathcal N$.   In our work we also observed that
 there is a surprising analogy between commutator widths and palindromic widths in groups, though we did not
 have a clear picture about the relationship. We asked \cite[Problem 2]{BG} a question to obtain a connection between the two widths. In this paper we also investigated this question and obtained further results in this direction, see \secref{cwpw}.

\medskip The main aim of this paper is to investigate the palindromic width of some solvable groups and the
wreath product of cyclic groups. We found a striking analogy of the palindromic widths and commutator widths of solvable groups.   Rhemtulla \cite[Theorem 1]{R1}  proved that if $A$ is a normal abelian subgroup of a solvable group $G$ such
that $G/A$ satisfies the maximal condition for normal subgroups, then ${\rm cw}(G) < \infty$. The following is an analogous theorem for palindromes.
\begin{theorem}\label{sol1}
Let $A$ be a normal abelian subgroup of a finitely generated solvable group $G=\langle X \rangle$ such that
$G/A$ satisfies the maximal condition for normal subgroups. Then $\pw(G, X) < \infty$.
\end{theorem}
We recall that a group $G$ is said to satisfy  \emph{the maximal condition for normal subgroups} if every normal subgroup of $G$ is the normal closure of a finite subset of $G$.
Our proof of the above theorem is based on the approach of Rhemtulla \cite{R1} where he studied commutator width in
finitely generated solvable groups. Essentially, we have been able to transfer the language of commutator
width into that of palindromes.  Rhemtulla \cite{R1} also obtained finiteness of the commutator width of
free $3$-step   solvable groups.  In this paper,  we obtain analogous results for palindromic width of free
$3$-step solvable groups. We further investigated the palindromic width of free abelian-by-nilpotent-by-nilpotent
groups.   We prove the following theorem in \secref{solv3}.
\begin{theorem}\label{unn}
Let $G$ be a finitely generated group in the variety  $\; \mathcal U \mathcal N \mathcal N$. Then the palindromic
width of $G$ is finite.
\end{theorem}
As a corollary we have:
\begin{cor}\label{sol3}
Every finitely generated $3$-step solvable group has finite palindromic width.
\end{cor}
We do not know whether the palindromic width of an arbitrary finitely generated solvable group
(of step $> 3$) is finite. Note, that an answer on the similar question on the commutator width also unknown.

\medskip In \cite{ BG2}, we have proved that the palindromic width of a finitely generated free
metabelian group of rank $n$ is at most $4n-1$. It would be interesting to obtain an exact value for the
palindromic width of a finitely generated metabelian group. It is indeed possible to obtain a precise
value of the palindromic widths for some metabelian groups. In this paper, we consider the special case of
the metabelian group $G=\Z \wr Z$.  If $a$ and $b$ are the generators of the first and the second
cyclic group in the above wreath product, then we prove  that $\pw(\Z \wr \Z, \{a, b\})=3$, see  \secref{wrs}. Along the way we also show that the commutator width of $\Z \wr \Z$ is $1$.

\medskip After finishing this work, we learned that Riley and Sale \cite{RS} have also studied palindromic
width of solvable and metableian groups using different techniques. They have proved the special case of \thmref{sol1} when $A$ is the trivial subgroup, that is, their result
states that the palindromic width of a finitely generated solvable group that satisfies the maximal condition
for normal subgroups has finite palindromic width. They have proved this using a result of
Akhavan-Malayeri \cite{AM1}. Riley and Sale have also proved that the palindromic width of
$\Z \wr \Z$ is $3$.  Their proof is based on their estimate of $\pw(G \wr Z^r)$, where $G$ is a finitely generated group. On the other hand, our proof relies on the fact (proved in \secref{wrs}) that
any element in the commutator subgroup of $\Z \wr \Z$ is a commutator. Palindromic width in wreathe products has also been investigated by Fink \cite{fink} who has also obtained an estimate of $\pw(G\wr \Z^r)$. Fink has also obtained the finiteness of palindromic widths of wreathe products of some more classes of groups.

\subsection*{Acknowledgements} We are thankful to Andrew Sale and Tim Riley
 for informing us about their work and for their comments on our work. Thanks are also due to Elisabeth Fink for her comments and interest on our work.

\section{Commutator Width and Palindromic Width} \label{cwpw}

 In this paper we shall use the following result by Rhemtulla \cite[Lemma 2]{R1}.
\begin{lemma}\label{reh1}
Let $G=\langle A, x_1, x_2, \ldots, x_n \rangle$ where $A$ is an abelian normal subgroup of $G$. Then the
subgroup $[A, G]$ is precisely the following set
$$[A, G]=\{[a_1, x_1][a_2, x_2]\ldots [a_n, x_n] ~ | ~ a_i \in A\}.$$
\end{lemma}
We shall also use the following theorem by Rhemtulla and Akhavan-Malayeri \cite{MR} to obtain a better bound for the metabelian groups.
\begin{theorem}\label{mr1}
Let $G=\langle x_1, \ldots, x_n \rangle$ is a non-abelian free abelian-by-nilpotent group of rank $n$. Let $A$ be an abelian normal subgroup of $G$ such that $G/A$ is nilpotent. Then every element $g \in G'$ can be expressed as:
$$g=[u_1, x_1]^{a_1} [u_2, x_2]^{a_2} \ldots [u_n, x_n]^{a_n}$$
for $u_1, \ldots, u_n \in G$ and $a_1, \ldots, a_n \in A$, where $u^b=b^{-1} u b$.
\end{theorem}

Also, we note the following lemma that is often useful to estimate palindromic width.
\begin{lemma}\label{onto} \cite{BG}
Let $G = \langle X \rangle$ and $H = \langle Y \rangle$ be two groups,
$\mathcal{P}(X)$ is the set of palindromes in the alphabet $X^{\pm 1},$
$\mathcal{P}(Y)$ is the set of palindromes in the alphabet $Y^{\pm 1}.$ If
$\varphi : G \longrightarrow H$ be an epimorphism such that
$\varphi(X) = Y$,  then
$$
{\rm pw}(H) \leq {\rm pw}(G).
$$
\end{lemma}

In \cite{BG} we also prove the following lemma.

\begin{lemma} \label{l1}
Let $G=\langle X \rangle$ be a group generated by a set $X$. Then the following
hold.
\begin{enumerate}
\item{ If $p$ is a palindrome, then for any integer $m$ the element $p^m$ is also a palindrome.}

\item{ Any element in $G$ which is conjugate to a product of $n$ palindromes, $n
\geq 1$,  is a product of $n$ palindromes if $n$ is even, and of $n+1$ palindromes
if $n$ is odd. }

\item{Any commutator of the type  $[u, p],$  where $p$ is a palindrome is
a product of $3$ palindromes.
Any element $[u, x^{\alpha}] x^{\beta}$, $x \in X$, $\alpha, \beta \in
\mathbb{Z},$ is a product of $3$ palindromes.}

\item{Any commutator of the type  $[u, p q],$  where $p, q$ are
palindromes is a product of $4$ palindromes.
Any element $[u, p x^{\alpha}] x^{\beta}$, $x \in X$, $\alpha, \beta \in
\mathbb{Z},$ is a product of $4$ palindromes.}
\end{enumerate}
\end{lemma}

The following is a generalization of (1) and (4) in the above lemma.

\begin{lemma}\label{g4}
Let $G=\langle X \rangle$ and $p_1, \ldots, p_k$ are palindromes in $G$ and $u \in G$. Then
\begin{enumerate}
\item{If $q = p_1 p_2$ is a product of two palindromes, then for any integer $m$ the element $q^m$ is also a
product of two palindromes.}

\item{
$$l_{\mathcal{P}}([u, p_1p_2\ldots p_k])\leq 2k+\varepsilon \hbox{ where } \varepsilon =
\left\{\begin{array}{ll}
0 & ~\mbox{if}~k\hbox{ is even}, \\
1 & ~\mbox{if}~k\hbox{ is odd}. \\
\end{array}
\right.
$$}
\end{enumerate}
\end{lemma}
\begin{proof}
(1) Let $m \geq 1$. Then $q^m = q^{m-1} p_1 \cdot p_2$ is a product of two palindromes. If $m \leq -1$, then
the assertion follows from the fact that the inverse to a palindrome is a palindrome.

(2) We see that
$$w=[u, p_1 p_2 \ldots p_k]=u^{-1}(p^{-1}_k \ldots p_2^{-1} p_1^{-1}) u~ p_1 p_2 \ldots p_k.$$
By \lemref{l1}(2) the element $u^{-1}(p^{-1}_k \ldots p_2^{-1} p_1^{-1}) u$ is a product of $k$ palindromes if $k$ is even and is a product of $k+1$ palindromes if $k$ is odd. Hence $w$ is a product of $2k + \varepsilon$ palindromes.
\end{proof}

Let $G=\langle X \rangle$ be a  group. Let $\bar G=\langle \bar X \rangle$, where $\bar X$
is the image of $X$ under some epimorphism of $G$. It follows from \lemref{onto} that
${\rm pw}(\bar G, \bar X) \leq {\rm pw}(G, X)$.  We now aim to find some upper bound for $\pw(G, X)$.

Let $H$ be a normal subgroup of $G$ such that $G/H=\bar G$. Then $H$ is generated by some elements
$h_1, h_2, \ldots$. The number of these elements may be infinite. However, any element $h \in H$ can be
represented by some word in the alphabet $X^{\pm 1}$. Denote by $\pw(H, X)$ the palindromic width of the
subgroup $H$ in the alphabet $X^{\pm 1}$.
\begin{lemma}\label{r1}
Let $G=\langle X \rangle$ be a group generated by $X$. Let $H$ be a normal subgroup of $G$ and $\bar G = G/H =
\langle \bar X \rangle$. Then
$$\pw(\bar G, \bar X) \leq \pw(G, X) \leq \pw(\bar G, \bar X) + \pw(H, X). $$
\end{lemma}
\begin{proof}
It follows from \lemref{onto} that $\pw(\bar G, \bar X) \leq \pw(G, X)$. To prove that $ \pw(G, X)
\leq \pw(\bar G, \bar X) + \pw(H, X)$, suppose that $\pw(\bar G, \bar X)=k$ and $\pw(H, X)=l$. Any $g \in G$
is a product $g=h \bar g$, where $h \in H$ and $\bar g$ is the coset representative of $G$ by $H$, i.e.
$g \in H \bar g$. Since $h$ is a product of at most $k$ palindromes and $\bar g$ is a product of at most $l$
palindromes, hence $g$ is a product of at most $k+l$ palindromes. This proves the lemma.
\end{proof}
\begin{cor}
Let $H=\langle Y \rangle$ be a group such that $\pw(H, Y) < \infty$. Let $G=\langle X \cup Y\rangle$ be a finite extension of $H$. Then $\pw(G, X \cup Y) < \infty$.
\end{cor}
 Note that \lemref{r1}  is a palindromic version of \cite[Lemma 1]{R1} where Rhemtulla proved similar
 inequalities for commutator widths. However, we do not know whether analogue of the above corollary holds for
 commutator widths. If $G=\langle X \rangle$ is a finitely generated group and $G$ belongs to some variety
 $\mathcal M$ of groups, then to prove that $\pw(G, X)$ finite, it follows from \lemref{onto} that it is enough
 to prove $\pw(F_n(\mathcal M))$ is finite.

In \cite[Problem 2]{BG} we asked a question on the connection of the palindromic width and the commutator width
of a group. The following proposition gives us this connection for some groups.
\begin{prop}\label{cp}
Let $G=\langle X \rangle$ be a finitely generated group with a set of generators $X = \{x_1, \ldots, x_n \}$ of cardinality $n$.
Suppose $G$ has a non-trivial center $Z=Z(G)$. Let $\bar G=G/Z=\langle \bar X \rangle$. Let ${\rm cw}(G)=l$
and $pw(\bar G, \bar X)=k$. Then
$$\pw(G, X) \leq n + l(2k+\varepsilon), \hbox{ where }$$ $$\varepsilon= \left\{\begin{array}{ll}
0 & ~\mbox{if}~k\hbox{ is even}, \\
1 & ~\mbox{if}~k\hbox{ is odd}. \\
\end{array}
\right.
$$
\end{prop}
\begin{proof}
We have the following exact sequence
$$1 \rightarrow Z \rightarrow G \xrightarrow{\pi} \bar G \rightarrow 1.$$
Let $s: \bar G \to G$ be a section. Then any element $g \in G$ has the form $g = z \tilde g$, where $z \in Z$
and $\tilde g=s(\bar g)$ for some $\bar g \in \bar G$. Since ${\rm cw}(G)=l$, any element $u \in G'=[G, G]$ has the form
$$u=[h_1, g_1] [h_2, g_2] \ldots [h_l, g_l], ~ h_i, ~g_i \in G \hbox{ for }i=1,2,\ldots, l.$$
For $i=1, 2, \ldots, l$, write any element $g_i$ in the form $g_i=z_i \tilde g_i$, $z_i \in Z, ~ \tilde g_i=s(\bar g_i)$.  Then
\begin{eqnarray*}
u&=& [h_1, z_1 \tilde g_1][h_2, z_2 \tilde g_2] \ldots [h_l, z_l \tilde g_l]\\
&=& [h_1, \tilde g_1][h_2,  \tilde g_2] \ldots [h_l, \tilde g_l].
\end{eqnarray*}
Note that arbitrary element $g \in G$ has the form
$$g=x_1^{\alpha_1} \ldots x_n^{\alpha_n} [h_1, \tilde g_1][h_2, \tilde g_2]\ldots [h_l, \tilde g_l],$$
for some integers $\alpha_1, \ldots, \alpha_n$.
By hypothesis, any element $\tilde g_i$ is a product of at most $k$ palindromes and by \lemref{g4}(2),
$[h_i, \tilde g_i]$
is a product of at most $2k+\varepsilon$ palindromes. Hence $u$ is a product of at most $l(2k+\varepsilon)$
palindromes.  Hence $g$ is a product of at most $n+l(2k+\varepsilon)$ palindromes. This proves the proposition.
\end{proof}

\section{Palindromic Width of Solvable Groups: Proof of \thmref{sol1}}\label{solv}

\begin{lemma}\label{lt1}
Let $G=\langle x_1, \ldots, x_n \rangle$ be a finitely generated group and
$H=\langle \langle y_1, \ldots, y_k \rangle \rangle$ be the normal closure of elements $y_1, \ldots, y_k$ in $G$.
Then $\pw(G/H') < \infty$ if and only if $\pw(G/H'')<\infty$.
\end{lemma}
\begin{proof}
Since $G/H'$ is a homomorphic image of $G/H''$, by \lemref{onto}, $\pw(G/H')<\infty$ if $\pw(G/H'')<\infty$.

Suppose that $\pw(G/H')<\infty$. We must prove that $\pw(G/H'')<\infty$.
We may assume $H''=1$. Then by \lemref{r1},
$$\pw(G) \leq \pw(G/[H', G] ) + \pw([H', G]).$$
By \lemref{reh1}
$$[H', G]=\{[h_1, x_1][h_2, x_2] \ldots[h_n, x_n]~|~  h_i \in H'\}.$$
 By \lemref{l1}(3) $\pw([H', G]) \leq 3n$. Hence
$$\pw(G) \leq \pw(G/[H', G]) + 3n.$$
We need to show only that $\pw(G/[H', G])<\infty$.

We may now assume that $[H', G]=1$. Thus $H' \leq Z(G)$ and $H$ is a 2-step nilpotent group. Hence the map
$h\mapsto [h, y_i], ~ h\in H$ is a homomorphism. Let
$$H_i=\{[h, y_i]~|~ h\in H\}.$$
Then $H_i \leq H'\leq Z(G)$ and $H_i$ is normal in $G$. So, $K=H_1 H_2\ldots H_k$ is a normal subgroup of $G$.
Thus the centralizer in $H$ of $H/K$:
$$Z_H(H/K)=\{h \in H ~|~[h, H] \subset K\}$$
is also normal in $G$. However, $y_i \in Z_H(H/K)$ for all $i=1, 2, \ldots, k$. Hence $H/K$ is abelian and
$H'\leq K$. Also, we have $K\leq H'$. This implies,  $K=H'$.

Every element of $K$ can be written in the form
$$
[h_1, y_1][h_2, y_2] \ldots[h_k, y_k], ~~  h_i \in H.
$$
Any $y_i$ is a word in the alphabet $X^{\pm 1}$. Let $m$ be the maximal length $l_X(y_i)$ of $y_i$ in the
alphabet $X^{\pm 1}$ for all
$i = 1, 2, \ldots, k$ then by \lemref{g4}(2) we have $l_{\mathcal{P}}([h_i, y_i]) \leq 2m + 1$.
Hence $\pw(H', X) \leq k(2m+1)$ and
$$\pw(G)\leq \pw(G/H')+k(2m+1).$$
This proves the lemma.
\end{proof}
\subsection{Proof of \thmref{sol1}}
\begin{proof}
Let $X=\{x_1, x_2, \ldots, x_n\}$. Take a minimal subset $Z \subseteq X$ such that $G=\langle A, Z\rangle$.
Without loss of generality, assume that $Z=\{ x_1, x_2, \ldots, x_m\}$, $m \leq n$. Let $M=\langle Z \rangle$
so that $G=AM$. By \lemref{r1},
\begin{equation} \label{eq1} \pw(G, X) \leq \pw(G/[A, G], \bar X) + \pw([A, G], X),\end{equation}
where $\bar X$ is the image of $X$ in $G/[A, G]$. It follows from \lemref{reh1} and \lemref{r1} that
$$\pw([A, G], X) \leq 3n.$$
Further, since we are considering $G/[A, G]$ in \eqnref{eq1}, without loss of generality we may further assume
that $A \leq Z(G)$. Thus we have
$$G'=[AM, AM]=[M, M]=M'.$$
It is therefore sufficient to show that $\pw(M, Z)$ is finite. Note that $M/M \cap A$ is isomorphic to $G/A$
and it satisfies the maximal condition for normal subgroups. Also $A \cap M \leq Z(M)$.

We will now complete the proof by using induction on the solvability length of $M$. If $M \in \mathcal U$,
i.e. $M$ is abelian, then it follows from \cite[Theorem 1]{BG} that $\pw(M, Z) \leq m$. Suppose the result
holds true for $M \in {\mathcal U}^r$. Consider the case when $M \in {\mathcal U}^{r+1}$. Let $L=M \cap A$.
Since $M/L$ satisfies the maximal condition for normal subgroups, there exists a subgroup
$N=\langle \langle y_1, y_2, \ldots, y_k \rangle \rangle$ of $M$ such that $N$ is the normal closure of elements $y_1, \ldots, y_k$ of $M$ and $LN=LM^{(r-1)}$. By induction $M^{(r+1)}=1$, so $(LN)'=N'=M^{(r)}$ is abelian.
Now by \lemref{lt1}, $\pw(M/N')<\infty$ implies that $\pw(M/N'')<\infty$. However, $N^{(2)}=M^{(r+1)}=1$ and
$M/N' \in \mathcal U^r$ and by hypothesis $M/N'$ has a finite palindromic width. Hence the theorem follows.
\end{proof}

\section{Proof of \thmref{unn}} \label{solv3}

Let $\mathcal M$ be some class of groups and $\mathcal N_c$ be the variety of nilpotent groups of nilpotency
class $\leq c$.  Clearly, for $c=1$, $\mathcal N_1=\mathcal U$. Let $\mathcal L_n$ be the class of groups generated by $n$ elements. Denote by $\mathcal L_n \cap \mathcal N_c \mathcal M$ the class of groups that belong to both $\mathcal L_n$ and $\mathcal N_c \mathcal M$.
\begin{lemma}\label{lt2}
For any class $\mathcal M$ of groups and any $n$, every group in  $\mathcal L_n \cap \mathcal N_2 \mathcal M$
has finite palindromic width if and only if every group in  $\mathcal L_n \cap \mathcal N_c \mathcal M$ has
finite palindromic width for $c \geq 2$.
\end{lemma}
\begin{proof}
If every group in  $\mathcal L_n \cap \mathcal N_c \mathcal M$, $c \geq 2$, has finite palindromic width, then
evidently every group in  $\mathcal L_n \cap \mathcal N_2 \mathcal M$ has finite palindromic width. We will use
induction on $c$ to prove the converse. If $c=2$ then there is nothing to prove. Assume the result holds for
$c =m-1$.  Let $G \in \mathcal L_n \cap \mathcal N_m \mathcal M$. Hence there is a normal subgroup $N$ of $G$
such that $G/N \in \mathcal M$ and $N \in \mathcal N_m$. Thus $\gamma_{m+1}(N)=1$. Since $m>2$,
$\gamma_{m}(N) = [\gamma_{m-1}(N), G]$ is abelian and hence by \lemref{r1}
$$\pw(G) \leq \pw(G/[\gamma_{m-1}(N), G]) + n.$$
Since $G/[\gamma_{m-1}(N), G] \in  \mathcal L_n \cap \mathcal N_{m-1} \mathcal M$, hence $\pw(G)<\infty$. This proves the lemma.
\end{proof}

\subsection{Proof of \thmref{unn}}
\begin{proof}
Let $G$ be a finitely generated group in $\mathcal U \mathcal N \mathcal N$. Then for some integers
$m, l$, $\gamma_m(\gamma_l(G))=A$ is abelian and normal in $G$. Thus by \lemref{r1} and \lemref{reh1},
$$\pw(G) \leq \pw(G/[A, G])+3n,$$
where $n$ is the number of generators of $G$. Now
$$\gamma_{m+1}(\gamma_l(G))=[A, \gamma_l(G)] \subset [A, G],$$
so that $G/[A, G] \in \mathcal L_n \cap \mathcal N_m \mathcal N_{l-1}.$ It is sufficient to show that any
finitely generated group in $\mathcal N \mathcal N$ has finite palindromic width. By \lemref{lt2}, it is enough
to show that any group in the class $\mathcal L_n \cap \mathcal N_2 \mathcal N$ has finite palindromic width.

We may now suppose that $G \in \mathcal L_n \cap \mathcal N_2 \mathcal N$ so that $G$ contains a normal
subgroup $M \in \mathcal N_2$ and $G/M \in \mathcal N$. Now $G/M'$ is finitely generated group in the class
$\mathcal U \mathcal N$. We know from \cite[Proposition 3.7]{BG} that the palindromic width of these groups is
bounded above by $5n$. Furthermore $M'$ is abelian and normal in $G$ and consequently has finite palindromic width. Hence $G$ has finite palindromic width.
\end{proof}

\section{ Palindromic and Commutator Widths of Wreath Products $\Z \wr \Z$} \label{wrs}
We have proven that if $F_n(\mathcal U^2)$ is a free metabelian group of rank $n$, then
$\pw(F_n(\mathcal U^2))\leq 4n-1.$ We think that this boundary is rough. For some metabelian groups we can
find the faithful value of palindromic width. For example, if
$$
BS(1,n) = \langle a, t~||~t^{-1} a t = a^n \rangle,~~n \in \mathbb{Z} \setminus \{0\}
$$
is a metabelian Baumslag-Solitar group, then we can prove

\begin{prop}
$\pw(BS(1,n), \{a, t \})=2$.
\end{prop}

\begin{proof}
It is not difficult to prove (see for example \cite[Lemma 11]{B1}) that every element $g \in BS(1,n)$ can
be presented in the form $g = t^k a^l t^{-m}$ for some non-negative integers $k, m$ and some integer $l$. Then
$$
g = t^k a^l t^k \cdot t^{-m-k}
$$
is the product of 2 palindromes. On the other side it is not difficult to prove, that element $t a$ is not
a palindrome.
\end{proof}

In this section, we aim to obtain palindromic width of $G=\Z \wr \Z$ where the palindromes are considered with
respect to the generating set $\{ a^{\pm 1}, b^{\pm 1}\}$, where $a$ is the generator of the first cyclic
group in the product and $b$ is the generator of the second cyclic group.
Let $A=\langle a \rangle \cong \Z$, $B=\langle b \rangle \cong \Z$. Then
$$G=\Z \wr \Z=A \wr B.$$
Let $a_0=a$ and $A_0=A=\langle a_0 \rangle$. Define the infinite set of groups $A_i=\langle a_i \rangle$ which are isomorphic copies of $A$ and $i \in \Z$. Let $C=\prod_{i \in Z}  A_i$ be the direct products of the $A_i$'s. We can identify the group $C$ as a subgroup of the group $F(B, A)$ of functions from $B$ to $A$,   where any function sends the element $b^i$ onto some element in $A_i$. If $f \in  C \subset F(B, A)$, then we will write
$$f=\prod_{i \in \Z} a_i^{n_i}, ~ n_i \in \Z, ~ f(b^i)=f(i)=a_i^{n_i},$$
where only finitely many numbers $n_i$ are non-zero. The group $B$ acts on $C$ by the rule
$$f^b=\prod_{i \in \Z} a_{i+1}^{n_i}.$$
This action induces an automorphism of $C$. Then
$G=A \wr B=C \ltimes B$,
where $\ltimes$ denotes the semi-direct product.

For any $f \in C$, define ${\rm supp}(f)$ to be the set of indices $i$ such that $n_i \ne 0$. For any $f \in C$, ${\rm supp}(f)$ is a finite set. Also we have the following conjugation rules for arbitrary integers $k$ and $l$:
$$b^{-k} A_l b^k=A_{l+k}.$$
\subsection{Commutator Subgroup of $G$}
\begin{lemma}\label{L1}
The commutator subgroup $G'=[G, G]$ is equal to the group $[C, B]$, where $[C, B]$ is a subgroup of $G$ generated by the commutators $[c, b^l]$ for $c \in C$, $l \in \Z$.
\end{lemma}
\begin{proof}
The inclusion $[C, B]\subset [G, G]$ is evident. To prove the reverse inclusion, consider any arbitrary commutator $[cb^k, c'b^l]$ in $[G, G]$, where $c, c' \in C, ~ k, l \in \Z$. Then
\begin{eqnarray*}
[cb^k, c'b^l] & = & b^{-k} c^{-1} b^{-l} (c')^{-1} c b^k c' b^l \\
&=& (c^{-1})^{b^k} b^{-k} b^{-l} (c')^{-1} c b^k b^l (c')^{b^l}\\
&=& c_k^{-1} (c')^{-b^{k+l}} c^{b^{k+l}} (c')^{b^l}, \hbox{ where }c_k=b^{-k} c b^{k}\\
&=& c_k^{-1} (c'_{k+l})^{-1} c_{k+1} c_l' \\
&=& c_k^{-1} c_{k+1} (c'_{k+l})^{-1} c'_l = [c_k, b][c'_{k+l}, b^{-k}].
\end{eqnarray*}
Hence $[G, G] \subset [C, B]$. This proves the lemma.
\end{proof}

\medskip We represent $c \in G$ as a product
 \begin{equation}\label{c1}
c=\prod_{i \in \Z}a_i^{n_i}, ~ n_i \in \Z, \end{equation}
 or as function $c: \Z \to \Z$ given by $c(i)=n_i$. Define a function $\log: C \to \Z$ as
$$\log c=\sum_{i \in \Z} n_i.$$
From \lemref{L1} we have the following.
\begin{cor}\label{wrc1}
The commutator subgroup $G'$ is contained in $C$. An element $c \in C$ belongs to $G'$ if and only if $\log c=0$.
\end{cor}

\begin{lemma}\label{t3}
Any element in the commutator subgroup $G'$ is a commutator.
\end{lemma}
\begin{proof}
Let $c \in G'$. We want to find an element $f \in C$ be such that $[f, b]=c$. To do this denote by $m$ the
minimal integer $m$ such that $n_m \neq 0$ in the representation of \eqnref{c1} and by $M$ the maximal
integer $M$ such that $n_M \neq 0$ in \eqnref{c1}. Hence the element $c$ has the form
$$c=a_m^{n_m} a_{m+1}^{n_{m+1}} \ldots a_{m+s-1}^{n_{m+s-1}} a_{m+s}^{n_{m+s}},$$
where $M=m+s$ and by \corref{wrc1},
\begin{equation} \label{e''} n_m + n_{m+1} + \ldots + n_{m+s-1} + n_{m+s}=0.\end{equation}
We will find the element $f$ in the form
$$f=a_m^{x_m} a_{m+1}^{x_{m+1}} \ldots a_{m+s-1}^{x_{m+s-1}}$$
with unknowns $x_m$, $x_{m+1}$, $\ldots$, $x_{m+s-1}$.
Then the commutator
\begin{eqnarray*}[f, b]=f^{-1} f^b&=&a_m^{-x_m} a_{m+1}^{-x_{m+1}} \ldots a_{m+s-1}^{-x_{m+s-1}} \cdot
 a_{m+1}^{x_m} a_{m+2}^{x_{m+1}} \ldots a_{m+s-1}^{x_{m+s-2}} a_{m+s}^{x_{m+s-1}}\\
&=& a_m^{-x_m}a_{m+1}^{x_m-x_{m+1}} a_{m+2}^{x_{m+1} -x_{m+2}} \ldots a_{m+s-1}^{x_{m+s-2} -x_{m+s-1}} a_{m+s}^{x_{m+s-1}}.
\end{eqnarray*}
Note that the equation $[f, b]=c$ gives us the following system of equations:
\begin{eqnarray*}
-x_m&=&n_m\\
-x_{m+1} + x_m&=&n_{m+1}\\
-x_{m+2} + x_{m+1} &=& n_{m+2} \\
\vdots\\
-x_{m+s-1} + x_{m+s-2}&=&n_{m+s-1} \\
x_{m+s-1} &=&  n_{m+s}.
\end{eqnarray*}
From the first $s$ equations we find
\begin{eqnarray*}
x_m&=& -n_m \\
x_{m+1} &=& x_m-n_{m+1} =-n_m -n_{m+1}, \\
x_{m+2} & = & x_{m+1} -n_{m+2}  =  -n_m -n_{m+1} -n_{m+2} ,\\
& & \vdots \\
x_{m+s-1} & = & x_{m+s-2} -n_{m+2} =-n_m -n_{m+1} -\cdots -n_{m+s-1}.
\end{eqnarray*}
Inserting this value of $x_{m+s-1}$ in the last equation we have
$$-n_m-n_{m+1} - \cdots-n_{m+s-1} = n_{m+s},$$
or
$$n_m + n_{m+1} + \cdots +n_{m+s-1} +n_{m+s} =0.$$
Since $c \in G'$ we see in \eqnref{e''} that this equality holds. So, we have a well-defined solution to the above system of equations and $f$ exists in $C$ such that $[f, b]=c$. This proves the theorem.
\end{proof}

\begin{cor} \label{wrcc1}
$\pw(G, \{a, b\}) \leq 3$.
\end{cor}
\begin{proof}
Every element in $G$ can be represented in the form
$$g=a^k b^l d, ~ d \in G', k, l \in \Z.$$
By Theorem \ref{t3}, there exists an element $f \in C$ such that $d=[f, b]$. Hence
$$g=a^k b^l [f, b]=a^k [f^{b^{-l}}, b]b^l.$$
Since elements in $C$ are commute then
$$
g=a^k [f^{b^{-l}}, b]b^l = a^k f^{-b^{-l}} b^{-1} f^{b^{-l}} b^{l+1}
= a^k f^{-b^{-l}} b^{-1} \left( \overline{f^{-b^{-l}}} \, \, \overline{f^{b^{-l}}} \right) f^{b^{-l}} b^{l+1} =
$$
$$
= f^{-b^{-l}} b^{-1} \overline{f^{-b^{-l}}} \cdot \overline{f^{b^{-l}}} a^k f^{b^{-l}} \cdot b^{l+1}.
$$
 Hence $g$ is a product of at most $3$ palindromes.
\end{proof}

\subsection{Palindromes in $G$}
The following proposition describes palindromes in $G$.
\begin{prop}\label{pg}
The palindromes in $G$ have one of the following forms:
$$a_{-k}^l f f^{b^{-2k}} b^{2k} \hbox{ or} ff^{b^{-l}} b^l, ~ k, l \in \Z.$$
\end{prop}
\begin{proof}
Since any element in $G$ has the form $fb^k$, $f \in C$, $k \in Z$, then any palindrome has one of the following forms
$$fb^ka^l \overline{fb^k}, ~ fb^k b^l \overline{f b^k}.$$
Since $C$ is an abelian group, it follows that $\overline{fb^k}=b^k f$. Hence
$$fb^k a^l \overline{fb^k}=fb^k a^l b^k f=f(a^l)^{b^{-k}} b^{2k} f=fa_{-k}^l f^{b^{-2k}} b^{2k} = a_{-k}^l f f^{b^{-2k}} b^{2k},$$
and
$$fb^k b^l \overline{fb^k} =fb^{2k+l} f = f f^{b^{-(2k+l)}} b^{2k +l}.$$
This proves the lemma.
\end{proof}
We will call palindromes of the first type by \emph{$a$-palindromes} and the palindromes of the second kind by $b$-palindromes. We see that we can find elements in $G$ that are not a product of $a$-palindromes, nor $b$-palindromes. It would be interesting to verify whether every element of $G$ is a product of an $a$-palindrome and a product of a $b$-palindrome.

\begin{theorem}\label{wrp}
$\pw(G, \{ a, b \}) = 3$.
\end{theorem}

The theorem will follow from the following proposition, \lemref{onto} and \corref{wrcc1}.
\begin{prop}
There is a homomorphism of the group $\Z \wr \Z$ onto $\N_{2,2}$.
\end{prop}
\begin{proof}

The group $\N_{2,2}$ has the following presentation:
\begin{equation} \label{n22} \N_{2,2} = \langle x, y \ | \ [y, x, x]=[y, x, y]=1 \rangle.\end{equation}
Consider the quotient $G/\gamma_3(G)$. We use the same letters $a$ and $b$ for the images of $a$ and $b$ in
$G/\gamma_3(G)$. Hence we have the following relations in $G/\gamma_3(G)$:
\begin{equation} \label{zzr} [b, a, a]=1, ~ [b, a, b]=1. \end{equation}

Since $[b, a]=b^{-1} a^{-1} b a=a_1^{-1} a = a_1^{-1} a_0$, we have
$[b, a, a]=[a_1^{-1} a_0, a_0]=1$, which is the  first relation in \eqnref{zzr}. The second relation has the form
$$[b, a, b]=[a_1^{-1} a_0, b]=a_0^{-1} a_1 b a_1^{-1} a_0 b = a_0^{-1} a_1 a_2^{-1} a_1 = a_0^{-1} a_1^2
a_2^{-1} =1, $$
i.e.
\begin{equation}\label{zzr2} a_0^{-1} a_1^2 a_2^{-1} =1.\end{equation}
We can write $a_2$ in the form
$$a_2=a_0^{-1} a_1^2.$$
Conjugating this relation by $b$ we have
$$a_3=a_1^{-1} a_2^2=a_1^{-1} (a_0^{-1} a_1^2)^2=a_0^{-2} a_1^3.$$
Conjugating this relation by $b$ we have
$$a_4=a_1^{-2} a_2^3=a_1^{-2} (a_0^{-1} a_1^2)^3=a_0^{-3} a_1^4.$$
Continuing this process we can remove all generators $a_2, a_3, \ldots, a_k, \ldots$  in the image of the group $C$ in $G/\gamma_3(G)$.
On the other hand, conjugating the relation \eqnref{zzr2} by $b^{-1}$ gives
$$a_{-2}=a_{-1} ^2 a_0^{-1} = (a_0^2 a_1^{-1} )^2 a_0^{-1} = a_0^3 a_1^{-2}.$$
Continuing this process we can remove all generators $a_{-1}, a_{-2}, \ldots, a_{-k}, \ldots$ in the image of
$C$ in $G/\gamma_3(G)$. Hence, any element $g$  in $G/\gamma_3(G)$ has the form
$$g=a_0^k a_1^l b^m=a^k b^{-1} a^l b b^m = a^k a^l (a^{-l} b^{-1} a^l b) b^m = a^{k+l} [a^l, b] b^m.$$
Since in $G/\gamma_3(G)$,  commutator $[a^l, b]=[a, b]^l$ commutes with $b^l$, thus
$$g=a^{k+l} b^m [b, a]^{-l}.$$
This shows that $G/\gamma_3(G)$ is isomorphic to $\N_{2,2}$. This proves the proposition.
\end{proof}

\section{ Open Questions} \label{oq}

We proved that for group $G$ in  Theorems \ref{sol1} and \ref{unn} the palindromic width
$\pw(G)$ is finite.

\begin{question}
Let $G$ be a group in  Theorems \ref{sol1} and \ref{unn}. Find the lower and upper bounds on the
palindromic width
$\pw(G)$.
\end{question}

In the last section we proved that the palindromic width of a solvable Baumslag-Solitar group is equal 2. For the
non solvable groups we can formulate the following.

\begin{question}
Let
$$
BS(m,n) = \langle a, t~||~t^{-1} a^m t = a^n \rangle,~~m, n \in \mathbb{Z} \setminus \{0\}
$$
be a non solvable Baumslag-Solitar group. Is it true that $\pw(BS(m,n), \{a, t \})$ is infinite?
\end{question}

Also, we can formulate the general question on the palindromic width of HNN-extensions.

\begin{question}
Let $G = \langle X \rangle$ be a  group and $A$ and $B$ are proper isomorphic subgroups of
$G$ and $\varphi : A \rightarrow B$ be an isomorphism. Is it true that the HNN-extension
$$
G^* = \langle G, t~||~t^{-1} A t = B, \varphi \rangle
$$
of $G$ with associated subgroups $A$ and $B$ has infinite palindromic width with respect to the generating set
$X \cup \{ t \}$?
\end{question}

In \cite{BT} palindromic widths of free products were investigated. For the generalized free products we can
formulate:

\begin{question}
Let $G = A*_C B$ be a free product of $A$ and $B$ with amalgamated subgroup $C$ and $|A : C| \geq 3$,
$|B : C| \geq 2$.
 Is it true that $\pw(G, \{A, B \})$ is infinite?
\end{question}

It would be interesting to understand palindromic width of wreath products. In particular, one can ask the following.

\begin{question}
Let $G = A \wr B$ be a wreath product of group $A = \langle X \rangle$ and $B = \langle Y \rangle$ such that
$\pw(A, X) < \infty$ and $\pw(B, Y) < \infty$.
 Is it true that $\pw(G, X \cup Y) < \infty$?
\end{question}

Riley and Sale \cite{RS} have proved that if $B$ is a finitely generated abelian group then the answer is
positive. Fink \cite{fink} has also proved the finiteness of palindromic width of wreath product in for some
more cases. However, the general case when $B$ is a finitely generated non-abelian group is still open.

\end{document}